\documentclass[11pt]{amsart}
\usepackage[utf8]{inputenc}
\usepackage{amsmath,amsthm,amsfonts,amssymb,mathtools}
\usepackage[margin=1.2in]{geometry}
\mathtoolsset{showonlyrefs}
\usepackage[color=orange!20, linecolor=orange, bordercolor=orange]{todonotes}

\usepackage{tikz}
\usetikzlibrary{shapes}
\tikzstyle{vertex}=[draw=black,circle,fill=black,minimum size=4pt, inner sep=0pt, outer sep=0pt,text=white,line width=0mm]
\tikzstyle{c0}=[shape=circle, minimum size=4pt, fill=white]
\tikzstyle{c1}=[shape=rectangle, minimum size=7pt, fill=red]
\tikzstyle{c2}=[shape=diamond, minimum size=10pt, fill=blue]
\newcommand{\cdiamond}{\tikz \node[vertex, c2, minimum size=.65em] at (0,0) {};}
\newcommand{\csquare}{\tikz \node[vertex, c1, minimum size=.5em] at (0,0) {};}

\newcommand{\C}{\mathcal{C}}
\renewcommand{\L}{\mathcal{L}}
\newcommand{\E}{\mathbb{E}}
\newcommand{\size}[1]{\lvert #1 \rvert}
\def\lam{\lambda}
\def\Lam{\Lambda}

\newcommand{\PC}[1][C]{P_{#1}}
\newcommand{\PCi}[2][C]{P_{#1}^{(#2)}}

\newcommand{\HWR}{H_{\text{WR}}}
\newcommand{\Hind}{H_{\text{ind}}}

\newtheorem{theorem}{Theorem}

\newtheorem{claim}[theorem]{Claim}
\newtheorem{lemma}[theorem]{Lemma}
\newtheorem{cor}[theorem]{Corollary}
\newtheorem{conj}[theorem]{Conjecture}

\title[On the Widom--Rowlinson occupancy fraction]{On the Widom--Rowlinson occupancy fraction in regular graphs}
\author{Emma Cohen, Will Perkins, Prasad Tetali}
\address{School of Mathematics, Georgia Institute of Technology}
\email{ecohen32@gatech.edu}
\address{School of Mathematics, University of Birmingham}
\email{math@willperkins.org}
\address{School of Mathematics, Georgia Institute of Technology}
\email{tetali@math.gatech.edu}
\thanks{Research of the first and the last authors is supported in part by the NSF grant DMS-1407657.}

\begin{document}

\begin{abstract}
We consider the Widom--Rowlinson model of two types of interacting particles on $d$-regular graphs. We prove a tight upper bound on the occupancy fraction, the expected fraction of vertices occupied by a particle under a random configuration from the model.  The upper bound is achieved uniquely by unions of complete graphs on $d+1$ vertices, $K_{d+1}$'s.  As a corollary we find that $K_{d+1}$ also maximises the normalised partition function of the Widom--Rowlinson model over the class of $d$-regular graphs.  A special case of this shows that the normalised number of homomorphisms from any $d$-regular graph $G$ to the graph $\HWR$, a path on three vertices with a loop on each vertex, is maximised by $K_{d+1}$. This proves a conjecture of Galvin. 
\end{abstract}
\subjclass[2010]{05C60, 05C35, 82B20}

\maketitle

\section{The Widom--Rowlinson Model}

A \emph{Widom--Rowlinson assignment} or \emph{configuration} on a graph $G$ is a map $\chi: V(G)\to \{0,1,2\}$ so that $1$ and $2$ are not assigned to neighbouring vertices, or in other words, a graph homomorphism from $G$ to the graph $\HWR$ consisting of a path on $3$ vertices with a loop on each vertex (the middle vertex represents the label $0$).  Call the set of all such assignments $\Omega(G)$.   The \emph{Widom--Rowlinson model} on $G$ is a probability distribution over $\Omega(G)$ parameterised by $\lam \in (0, \infty)$, given by: 
\begin{align}
    {\mathbb P}[\chi] &= \frac{\lambda^{X_1(\chi) + X_2(\chi)}}{P_G(\lambda)},
\end{align}
where $X_i(\chi)$ is the number of vertices coloured $i$ under $\chi$, and 
\begin{align}
P_G(\lambda) &= \sum_{\chi \in \Omega(G)} \lambda^{X_1(\chi) + X_2(\chi)} 
\end{align}
 is the \textit{partition function}.  Evaluating $P_G(\lam)$ at $\lam=1$ counts the number of homomorphisms from $G$ to $\HWR$. We think of vertices assigned $1$ and $2$ as ``coloured'' and those assigned $0$ as ``uncoloured'' (see Figure \ref{fig:WR-config}). 

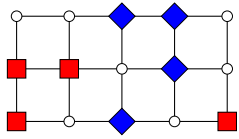
\begin{figure}[t]
\centering
\begin{tikzpicture}[scale=.7, baseline=.5cm]
    \foreach \x in {0,...,4}{
        \foreach \y in {0,1,2}{
            \node[vertex, c0] (v\x\y) at (\x,\y) {};
        }
        \draw (v\x0) to (v\x1) to (v\x2);
    }
    \draw (v00) to (v10) to (v20) to (v30) to (v40)
          (v01) to (v11) to (v21) to (v31) to (v41)
          (v02) to (v12) to (v22) to (v32) to (v42);
    \foreach \c in {00, 01, 11, 40}{
        \node[vertex, c1] at (v\c) {};
    }
    \foreach \c in {20, 22, 31, 32}{
        \node[vertex, c2] at (v\c) {};
    }
\end{tikzpicture}
\caption{A configuration for the Widom--Rowlinson model on a grid. Vertices mapping to 1 and 2 are shown as squares and diamonds, respectively (corresponding to Figure \ref{fig:HWR}).}
\label{fig:WR-config}
\end{figure}

The Widom--Rowlinson model was introduced by Widom and Rowlinson in 1970 \cite{widom1970new}, as a model of two types of interacting particles with a hard-core exclusion between particles of different types: colour $1$ and $2$ represent particles of each type and colour $0$ represents an unoccupied site.  The model has been studied both on lattices \cite{lebowitz1971phase} and in the continuum \cite{ruelle1971existence,chayes1995analysis} and is known to exhibit a phase transition in both cases. 

The Widom--Rowlinson model is one case of a general random model: that of choosing a random homomorphism from a large graph $G$ to a fixed graph $H$.  In the Widom--Rowlinson case, we take $H=\HWR$. Another notable case is $\Hind$, an edge between two vertices, one of which has a loop (see Figure \ref{fig:HWR}).  Homomorphisms from $G$ to $\Hind$ are exactly the independent sets of $G$, and the partition function of the hard-core model is the sum of $\lam^{|I|}$ over all independent sets $I$.   An overview of the connections between statistical physics models with hard constraints, graph homomorphisms, and combinatorics can be found in \cite{winkler2002bethe}.    

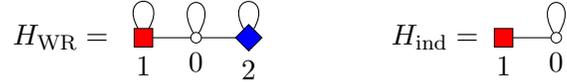
\begin{figure}
\centering
\[\HWR = 
\begin{tikzpicture}[scale=.7, baseline=-1mm]
  \node[vertex, c1, label={below:$1$}] (v1) at (-1,0) {};
  \node[vertex, c0, label={below:$0$}] (v0) at (0,0) {};
  \node[vertex, c2, label={below:$2$}] (v2) at (1,0) {};
  \draw (v1) .. controls +(60:1) and +(120:1) .. (v1) to
        (v0) .. controls +(60:1) and +(120:1) .. (v0) to
        (v2) .. controls +(60:1) and +(120:1) .. (v2);
\end{tikzpicture}
\qquad
\qquad
\Hind = 
\begin{tikzpicture}[scale=.7, baseline=-1mm]
  \node[vertex, c1, label={below:$1$}] (v1) at (-1,0) {};
  \node[vertex, c0, label={below:$0$}] (v0) at (0,0) {};
  \draw (v1) to (v0) .. controls +(60:1) and +(120:1) .. (v0);
\end{tikzpicture}
\]
\caption{The target graphs for the Widom--Rowlinson model and the hard-core model.}
\label{fig:HWR}
\end{figure}

For every such model, there is an associated extremal problem.  Denote by $\mathrm{hom}(G,H)$ the number of homomorphisms from $G$ to $H$. Then we can ask which graph $G$ from  a class of graphs $\mathcal G$ maximises $\mathrm{hom}(G,H)$, or if we wish to compare graphs on different numbers of vertices, ask which graph maximises the scaled quantity $\mathrm{hom}(G,H)^{1/|V(G)|}$.  

 Kahn  \cite{kahn2001entropy} proved that for any $d$-regular, bipartite graph $G$,   
 \begin{align}
 \label{kahnThm}
  \mathrm{hom}(G,\Hind) &\le \mathrm{hom}(K_{d,d},\Hind)^{|V(G)|/2d} \,, 
  \end{align}
  where $K_{d,d}$ is the complete $d$-regular bipartite graph. Equality holds in \eqref{kahnThm} if $G$ is $K_{d,d}$ or a union of $K_{d,d}$'s.  In other words, unions of $K_{d,d}$'s maximise the total number of independent sets over all $d$-regular, bipartite graphs on a fixed number of vertices.
   
   In a broad generalisation of Kahn's result, Galvin and Tetali \cite{galvin2004weighted} showed that in fact, \eqref{kahnThm} holds for all $d$-regular, bipartite $G$ and \textit{all} target graphs $H$ (including, for example, $\HWR$).  And using a cloning construction and a limiting argument, they showed that in fact the partition function of such models (a weighted count of homomorphisms) is maximised by $K_{d,d}$; for example, for a $d$-regular, bipartite $G$,
 \begin{equation}
 \label{galTetthm}
  P_G(\lam) \le P_{K_{d,d}}(\lam)^{|V(G)|/2d},
  \end{equation}
 where $P_G(\lam)$ is the Widom--Rowlinson partition function defined above or the independence polynomial of a graph.  Note that the case $\lam =1$ is the counting result.  
 
 There is no such sweeping statement for the class of all $d$-regular graphs with the bipartiteness restriction removed.   In \cite{zhao2010number} and \cite{zhao2011bipartite}, Zhao showed that the bipartiteness restriction on $G$ in \eqref{kahnThm} and \eqref{galTetthm} can be removed for some class of graphs $H$, including $\Hind$.  But such an extension is not possible for all graphs $H$; for example, $K_{d+1}$ has more homomorphisms to $\HWR$ than does  $K_{d,d}$ (after normalising for the different numbers of vertices).  In fact Galvin conjectured the following:

\begin{conj}[Galvin \cite{galvin2013maximizing,galvin2014three}]
\label{conj:glavin}
Let $G$ be a any $d$-regular graph.  Then
\[ \mathrm{hom}(G,\HWR) \le \mathrm{hom}(K_{d+1},\HWR)^{|V(G)|/(d+1)} \, .\]
\end{conj}

The more general Conjecture 1.1 of \cite{galvin2013maximizing} that the maximising $G$ for any $H$ is either $K_{d,d}$ or $K_{d+1}$ has been disproved by Sernau \cite{sernau2015graph}.

The above theorems of Kahn and Galvin and Tetali are based on the \textit{entropy method} (see \cite{radhakrishnan20036} and \cite{galvin2014three} for a survey), but in this context bipartiteness seems essential for the effectiveness of the method.  We will approach the problem differently, using the \textit{occupancy method} of \cite{davies2015independent}.

We first define the \emph{occupancy fraction} $\alpha_G(\lam)$ to be the expected fraction of vertices which receive a (nonzero) colour in the Widom--Rowlinson model:
\[ \alpha_G(\lam) = \frac{\E [X_1 + X_2]}{|V(G)|}  \, , \]
where $X_i$ is the number of vertices coloured $i$ by the random assignment $\chi$.  A calculation shows that $\alpha_G(\lam)$ is in fact the scaled logarithmic derivative of the partition function:
\begin{equation}
\label{eqalphadef}
 \alpha_G(\lam) = \frac{\lam}{|V(G)|} \cdot \frac{ P_G^\prime(\lam)}{P_G(\lam)} = \frac{\lam \cdot ( \log P_G(\lam))^\prime}{|V(G)|}  \, .
 \end{equation}

Our main result is that for any $\lam$,  $\alpha_G(\lam)$  is maximised over all $d$-regular graphs $G$ by $K_{d+1}$.

\begin{theorem} \label{thm:main}
    Let $G$ be any $d$-regular graph and $\lam > 0$. Then 
    \[ \alpha_G(\lam) \le \alpha_{K_{d+1}}(\lam) \]
    with equality if and only if $G$ is a union of $K_{d+1}$'s. 
\end{theorem}

We will prove this by introducing local constraints on random configurations induced by the Widom--Rowlinson model on a $d$-regular graph $G$, then solving a linear programming relaxation of the optimisation problem over all $d$-regular graphs.

Theorem~\ref{thm:main} implies maximality of the normalised partition function:
\begin{cor} \label{cor:partition}
    Let $G$ be a $d$-regular graph and $\lam > 0$. Then 
    \[ \frac{1}{|V(G)|} \log P_G(\lam) \le \frac{1}{d+1} \log P_{K_{d+1}}(\lam) \,,  \]
    or equivalently,
    \[ P_G(\lam) \le P_{K_{d+1}}(\lam)^{|V(G)|/(d+1)} \, ,\]
    with equality if and only if $G$ is a union of $K_{d+1}$'s. 
\end{cor}
The quantity $ \frac{1}{|V(G)|} \log P_G(\lam) $ is known in statistical physics as the \textit{free energy per unit volume}.  
Corollary~\ref{cor:partition} follows from Theorem~\ref{thm:main} as follows:  $\frac{1}{|V(G)|} \log P_G(0) =0$ for any $G$, and so
\begin{align*}
    \frac{1}{|V(G)|} \log P_G(\lam) 
    &= \frac{1}{|V(G)|} \int_0 ^\lam  \left(  \log P_G(t) \right )^\prime \, dt \\
    &\le \frac{1}{d+1} \int_0^\lam   \left(  \log P_{K_{d+1}}(t) \right )^\prime \, dt = \frac{1}{d+1} \log P_{K_{d+1}}(\lam) 
\end{align*}
where the inequality follows from Theorem~\ref{thm:main} and \eqref{eqalphadef}.  Exponentiating both sides gives Corollary~\ref{cor:partition}.

By taking $\lam=1$ in Corollary~\ref{cor:partition}, we get the counting result:
\begin{cor}
For all $d$-regular $G$,
\[ \mathrm{hom}(G,\HWR) \le \mathrm{hom}(K_{d+1},\HWR)^{|V(G)|/(d+1)} \]
with equality if and only if $G$ is a union of $K_{d+1}$'s.
\end{cor}
This proves Conjecture \ref{conj:glavin}. 

\subsection*{Discussion and related work}
\label{sec:related}

The method we use is more probabilistic than the entropy method in the sense that Theorem~\ref{thm:main} gives information about an observable of the model; in some statistical physics models, the analogue of $\alpha_G(\lam)$ would be called the \emph{mean magnetisation}.  We also work directly in the statistical physics model, instead of counting homomorphisms.  

Davies, Jenssen, Perkins, and Roberts \cite{davies2015independent} applied the occupancy method to two central models in statistical physics: the hard-core model of a random independent set described above, and the monomer-dimer model of a randomly chosen matching from a graph $G$.  In both cases they showed that $K_{d,d}$ maximises the occupancy fraction over all $d$-regular graphs. In the case of independent sets this gives a strengthening of the results of Kahn, Galvin and Tetali, and Zhao, while for matchings, it was not known previously that unions of $K_{d,d}$ maximises the partition function or the total number of matchings.

The idea of calculating the log partition function by integrating a partial derivative is not new of course; see for example, the interpolation scheme of Dembo, Montanari, and Sun \cite{dembo2013factor} in the context of Gibbs distributions on locally tree-like graphs.  The method is powerful because it reduces the computation of a very global quantity, $P_G(\lam)$, to that of a locally estimable quantity, $\alpha_G(\lam)$. 

Some partial results towards the Widom--Rowlinson counting problem were obtained by Galvin \cite{galvin2013maximizing}, who showed that a graph with more homomorphisms than a union of $K_{d+1}$'s must be close in a specific sense to a union of  $K_{d+1}$'s.

\section{Proof of Theorem \ref{thm:main}}
\label{sec:proof}

\subsection{Preliminaries}
To prove Theorem \ref{thm:main}, we will use the following experiment: for a $d$-regular graph $G$, we first draw a random $\chi$ from the Widom--Rowlinson model, then select a vertex $v$ uniformly at random from $V(G)$. We then write our objective function, the occupancy fraction, in terms of local probabilities with respect to this experiment, and add constraints on the local probabilities that must hold for all $G$.  We then relax the optimisation problem to all distributions satisfying the local constraints, and optimise using linear programming.    

Fix $d$ and $\lambda$. Define a \emph{configuration with boundary conditions} $C =(H, \mathcal L)$ to be a graph $H$ on $d$ vertices with family of lists $\L = \{ L_u \}_{u \in H}$, where each $L_u \subseteq \{1,2\}$ is a set of allowed colours for the vertex $u$.  Here $H$ represents the neighbourhood structure of a vertex $v\in V(G)$ and the colour lists $L_u$ represent the colours permitted to neighbours of $v$, given an assignment $\chi$ on the vertices outside of $N(v)\cup \{v\}$. (See Figure \ref{fig:config_C}.)  Denote by $\C$ the set of all possible configurations with boundary conditions in any $d$-regular graph.

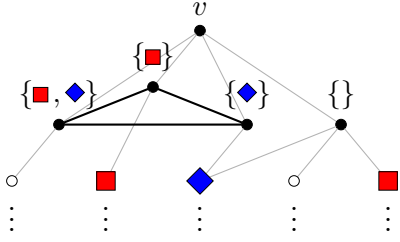
\begin{figure}
\centering
\begin{tikzpicture}[xscale=1.25]
  \node[vertex, label={above:$v$}] (v) at (0,0) {};
  \draw[draw=black!30] 
                (v) to (-1.5,-1.25) node[vertex, label={above:$\{\csquare\,,\cdiamond\}$}] (u1) {}
                (v) to (-.5, -0.75) node[vertex, label={above:$\{\csquare\}$}] (u2) {}
                (v) to (0.5, -1.25) node[vertex, label={above:$\{\cdiamond\}$}] (u3) {}
                (v) to (1.5, -1.25) node[vertex, label={above:$\{\}$}] (u4) {};
  \foreach \i in {1,...,5}{
    \coordinate[label=below:\vdots] (x\i) at ({\i-3},-2) {};}
  \draw[draw=black!30]
        (u4) to (x5)
        (u4) to (x4)
        (u4) to (x3)
        (u3) to (x3)
        (u2) to (x2)
        (u1) to (x1);
  \draw[thick] (u3) to (u1) to (u2) to (u3);
  \node[vertex, c0] at (x1) {};
  \node[vertex, c1] at (x2) {};
  \node[vertex, c2] at (x3) {};
  \node[vertex, c0] at (x4) {};
  \node[vertex, c1] at (x5) {};
\end{tikzpicture}
\caption{An example configuration with boundary conditions based on a colouring $\chi$. The graph $H$ consists of the four neighbours of $v$ along with the black edges, and the list $L_u$ is shown above each vertex $u$ of $H$. The colours assigned by $\chi$ to $v$ and its neighbours are immaterial and so are not shown.}
\label{fig:config_C}
\end{figure}

We now pick the assignment $\chi$ at random from the Widom--Rowlinson model on a fixed $d$-regular graph $G$, pick a vertex $v$ uniformly at random from $V(G)$, and consider the probability distribution induced on $\C$.

For example, if $G = K_{d+1}$ then with probability 1 the random configuration $C$ is $H = K_d$ with $L_u = \{1,2\}$ for all $u\in V(H)$. If $G = K_{d,d}$ then $H$ is always $d$ isolated vertices and the colour lists can be any (possibly empty) subset of $\{1,2\}$, but the lists must be the same for all $u\in V(H)$.

For a configuration $C = (H, \L)$, define 
\begin{align}
    \alpha_i^v(C) &= {\mathbb P}[\chi(v) = i \mid C]\\
    \alpha_i^u(C) &= \frac{1}{d} \sum_{u\in V(H)} {\mathbb P}[\chi(u) = i \mid C],
\end{align}
where the probability is over the Widom--Rowlinson model on $G$  given the boundary conditions $\mathcal L$. Note that the spatial Markov property of the model means that these probabilities are ``local'' in the sense that they can be computed knowing only $C$. Let $\alpha^v(C) = \alpha_1^v(C) + \alpha_2^v(C)$ and $\alpha^u(C) = \alpha_1^u(C) + \alpha_2^u(C)$. Then we have
\begin{align}
\label{eqAlpEq}
 \alpha_G(\lam) &= \frac{1}{|V(G)|} \sum_{v \in V(G)} {\mathbb P}[ \chi(v)\in \{1,2\}] =\E_C [ \alpha^v(C)] \\
 &=\frac{1}{d} \frac{1}{|V(G)|} \sum_{v \in V(G)} \sum_{u \sim v} {\mathbb P}[ \chi(u)\in \{1,2\}]  = \E_C[ \alpha^u(C)] \, ,
 \end{align}
where the expectations are over the probability distribution induced on $\C$ by our experiment of drawing $\chi$ from the model and $v$ uniformly at random from $V(G)$, and the last sum is over all neighbours of $v$ in $G$. Equality of the two expressions for $\alpha$ follows since sampling a uniform neighbour of a uniform vertex in a regular graph is equivalent to sampling a uniform vertex.  We will show that this expectation is maximised when the graph $G$ is $K_{d+1}$.

We can in fact write explicit formulae for $\alpha^v(C)$ and $\alpha^u(C)$. For a configuration $C = (H, \L)$, let $\PCi{0}(\lambda)$ be the total weight of colourings of $H$ satisfying the boundary conditions given by the lists $\L$ (corresponding to the partition function for the neighbourhood of $v$ conditioned on $\chi(v) = 0$). Also, write $\PCi{i}(\lambda)$ for the total weight of colourings of $H$ satisfying the boundary conditions and using only colour $i$ and $0$ (corresponding to the partition functions for the neighbourhood of $v$ conditioned on $\chi(v) = i$). Finally, let $\PCi{12}(\lambda) = \PCi{1}(\lambda) + \PCi{2}(\lambda)$ and let 
\begin{align}
\PC(\lambda) = \PCi{0}(\lambda) + \lambda \PCi{12}(\lambda)
\end{align}
be the partition function of $N(v)\cup\{v\}$ conditioned on the boundary conditions given by $C$. Note that if $\L$ has $a_1$ lists containing $1$ and $a_2$ lists containing $2$, then $\PCi{i}(\lambda) = (1+\lambda)^{a_i}$.

Now we can write
\begin{align}
\label{alphaFormulas}
    \alpha^v(C) &= \frac{\lambda \PCi{12}}{\PC} & \text{and}&&
    \alpha^u(C) &= \frac{\lambda\Big( (\PCi{0})^\prime + \lambda (\PCi{12})^\prime \Big)}{d\,\PC}\,,
\end{align}
where $P^\prime$ is the derivative of $P$ in $\lambda$. We will suppress the dependence of the partition functions on $\lam$ from now on. 

For $G = K_{d+1}$, we have
\begin{align}
    P_{K_{d+1}}&= 2(1 + \lambda)^{d+1} - 1\\
    \alpha_{K_{d+1}}(\lam) &=   \frac{2\lambda(1+\lambda)^d}{2(1+\lambda)^{d+1} - 1}\,.
\end{align}
If $G = K_{d+1}$ then the only possible configuration is  $C_{K_{d+1}}$, the complete neighbourhood $K_d$ with full boundary lists, so we also have $\alpha^u(K_d) = \alpha^v(K_d) = \alpha_{K_{d+1}}(\lam)$ (we can also compute these directly). Since this quantity will arise frequently, we will use the notation $\alpha_K = \alpha_{K_{d+1}}(\lam)$.

\subsection{A linear programming relaxation}

Now let $q: \C\to [0,1]$ denote a probability distribution over the set of all possible configurations. Then we set up the following optimisation problem over the variables $q(C),\, C\in \C$.
\begin{align}\label{primalProgram}
    \alpha^* &= \max \sum_{C\in \C} q(C) \alpha^v(C) \qquad \text{subject to}\\
    & \sum_{C\in \C} q(C) = 1 \\
    & \sum_{C\in \C} q(C) [\alpha^v(C) - \alpha^u(C)] = 0 \\
    & q(C) \geq 0 \,\,\,\, \forall C\in \C.
\end{align}

Note that this linear program is indeed a relaxation of our optimisation problem of maximising $\alpha_G(\lam)$ over all $d$-regular graphs: any such graph induces a probability distribution on $\C$, and as we have seen above in \eqref{eqAlpEq}, the constraint asserting the equality  $\E \alpha^v(C) = \E \alpha^u(C)$ must hold in all $d$-regular graphs. 

We will show that for any $\lam >0$ the unique optimal solution of this linear program is $q(C_{K_{d+1}}) =1$, where $C_{K_{d+1}}$ is the configuration induced by $K_{d+1}$: $H= K_d$ and $L_u =\{1,2\}$ for all $u\in H$. 

The dual of the above linear program is
\begin{align}\label{dualProgram}
    \alpha^* &= \min \Lambda_p \qquad \text{subject to}\\
    & \Lambda_p + \Lambda_c(\alpha^v(C) - \alpha^u(C)) \geq \alpha^v(C) & \forall C\in \C,
\end{align}
with decision variables $\Lambda_p$ and $\Lambda_c$. 

To show that the optimum is attained by $C_{K_{d+1}}$, we must find a feasible solution to the dual program with $\Lambda_p = \alpha_K= \frac{2\lambda(1+\lambda)^d}{2(1+\lambda)^{d+1} - 1}$. Note that with $\Lambda_p = \alpha_K$ the constraint for $C_{K_{d+1}}$  holds with equality for any choice of $\Lambda_c$. In other words, it suffices to find some convex combination of the two local estimates $\alpha^u$ and $\alpha^v$ which is maximised by $C_{K_{d+1}}$ over all $C\in \C$.

Let $C_0$ be a configuration with $L_u = \emptyset$ for all $u\in H$ (in which case the edges of $H$ are immaterial, and so abusing notation we will refer to any one of these configurations as $C_0$). We find a candidate $\Lam_c$ by solving the constraint corresponding to $C_0$ with equality:
\begin{align*}
\alpha_K &= \Lambda_c(\alpha^u(C_0) - \alpha^v(C_0)) + \alpha^v(C_0) \\
&= (1- \Lam_c )\frac{2\lam}{1+ 2\lam} \, .
\end{align*}
This gives
\begin{align*}
\Lam_c &= 1 - \frac{\alpha_K}{2\lambda} (1+2\lambda) = \frac{\alpha_K}{2\lambda} \frac{(1+\lambda)^d - 1}{(1+\lambda)^d} \, .
\end{align*}
With this choice of $\Lam_c$, the general dual constraint is
\begin{align}
    \alpha_K &\geq \frac{\alpha_K}{2\lambda} \frac{(1+\lambda)^d-1}{(1+\lambda)^d} \alpha^u(C) + \frac{\alpha_K}{2\lambda} (1+2\lambda) \alpha^v(C) \, .
\end{align}
Using \eqref{alphaFormulas}, this becomes 
\begin{align}
\label{manipConstraint}
\frac{(\PCi{0})^\prime + \lambda (\PCi{12})^\prime}{2\PCi{0} - \PCi{12}} &\leq \frac{d(1+\lambda)^d}{(1+\lambda)^d-1}\,.
\end{align}
From this point on we may assume that $C$ has some non-empty colour list, since otherwise the configuration is equivalent to  $C_0$ and the constraint holds with equality by our choice of $\Lam_c$. This assumption tells us, among other things, that $(\PCi{0})^\prime > 0$ and $2\PCi{0} - \PCi{12} > 0$.

Our goal is now to show that \eqref{manipConstraint} holds for all $C$.  We consider the two terms separately.  

\begin{claim}
\label{claim:term2}
For any $C \ne C_0$,
\[  \frac{\lambda (\PCi{12})^\prime }{2\PCi{0} - \PCi{12}} \le \frac{d \lam (1+\lam)^{d-1}}{(1+\lam)^d -1}  \, ,\]
with equality if and only if the lists $L_u$ are all equal and $C$ has no dichromatic colourings.
\end{claim}
\begin{proof}
Since the partition function $\PCi{0}$ is at least the total weight $\PCi{1} + \PCi{2} - 1$ of monochromatic colourings (with equality when $C$ has no dichromatic colourings), we have 
\begin{align}
    \frac{(\PCi{12})^\prime}{2\PCi{0} - \PCi{12}}
    \leq \frac{(\PCi{12})^\prime}{\PCi{12} - 2}
    = \frac{a_1 (1+\lambda)^{a_1-1} + a_2 (1+\lambda)^{a_2-1}}{(1+\lambda)^{a_1} + (1+\lambda)^{a_2} - 2}
\end{align}
(where, as above, $a_i$ is the number of vertices in $H$ allowed colour $i$ under the given boundary conditions), and so we need to show that 
\begin{align}
\label{eqlama1a2}
\frac{a_1 (1+\lambda)^{a_1-1} + a_2 (1+\lambda)^{a_2-1}}{(1+\lambda)^{a_1} + (1+\lambda)^{a_2} - 2} &\le \frac{d (1+\lam)^{d-1}}{(1+\lam)^d -1}  \, .
\end{align}
In general, to show that $(a+b)/(c+d) \leq t$ it suffices to show that $a/c \leq t$ and $b/d \leq t$. Thus it is enough to show that
\begin{equation}
\label{singleLameq}
\frac{a (1+\lambda)^{a - 1}}{(1+\lambda)^a - 1} \leq \frac{d (1+\lambda)^{d-1}}{(1+\lambda)^d - 1}
\end{equation}
whenever $1\leq a \leq d$. (Note that if either $a_1=0$ or $a_2=0$ then \eqref{eqlama1a2} reduces to \eqref{singleLameq}, and if both $a_1, a_2 =0$ then the configuration is $C_0$).   Indeed, it is not hard to check via calculus that the left hand side of \eqref{singleLameq} is increasing with $a$. This completes the proof of the inequality in Claim~\ref{claim:term2}.

We have equality in this final step when $a_1 = a_2 = d$ or when one is 0 and the other is $d$. So we have equality overall whenever the lists are all equal and there are no dichromatic colourings  (recall that we are assuming $C$ has some non-empty colouring list).
\end{proof}

\begin{claim}
\label{claim:term1}
For any $C \ne C_0$,
\[ \frac{(\PCi{0})^\prime  }{ 2\PCi{0} - \PCi{12} }   \le \frac{d(1+\lam)^{d-1} }{(1+\lam)^d -1  }\, ,\]
with equality if and only if the lists $L_u$ are all equal and $C$ has no dichromatic colourings.
\end{claim}

\begin{proof}
We can write
\begin{align}
    \frac{\lambda (\PCi{0})^\prime}{2\PCi{0}- \PCi{12}} &= \frac{\lambda(\PCi{0})^\prime}{\PCi{0}} \cdot \frac{\PCi{0}}{(\PCi{0} - \PCi{1}) + (\PCi{0} - \PCi{2})}\\
    &= \frac{\E_C[X_1] + \E_C[X_2]}{{\mathbb P}_C[X_1 > 0] + {\mathbb P}_C[X_2 > 0]},
\end{align}
where now $X_i$ is the number of vertices coloured $i$ in a random colouring chosen from the Widom--Rowlinson model on $C$.  Noting that $\E_C[X_1] = 0$ whenever ${\mathbb P}_C[X_1 > 0] = 0$, it suffices as above to show that whenever colour $1$ is permitted anywhere in $C$,
\begin{align} \label{cond_exp}
    \frac{\E_C[X_1]}{{\mathbb P}_C[X_1 > 0]} = \E_C[X_1\ |\ X_1 > 0]
    \leq \frac{\lambda d (1+\lambda)^{d-1}}{(1+\lambda)^d - 1} = \E_{K_d}[X_1\ |\ X_1 > 0] \, ,
\end{align}
and similarly for $X_2$, but this will follow by symmetry.  

We can decompose the expectation as
\begin{align}
    \E_C[X_1\ |\ X_1 > 0]
    &= \sum_{S\subseteq V(H)} {\mathbb P}_C[\chi^{-1}(2) = S\ |\ X_1 > 0]\cdot \E_C[X_1\ |\ X_1 > 0 \land \chi^{-1}(2) = S] \,.
\end{align}
The partition function restricted to colourings satisfying $X_1 > 0$ and $\chi^{-1}(2) = S$ is just $P_S(\lambda) = \lambda^{\size{S}} ((1+\lambda)^{a_S} - 1)$, where $a_S$ is the number of vertices in $H\setminus S$ which are allowed colour $1$ and are not adjacent to any vertex of $S$. The conditional expectation is then
\begin{align}
    \E_C[X_1\ |\ X_1 > 0 \land \chi^{-1}(2) = S]
    &= \frac{a_S \lambda (1+\lambda)^{a_S - 1}}{(1+\lambda)^{a_S} - 1}
    \leq \frac{d \lambda (1+\lambda)^{d-1}}{(1+\lambda)^d - 1}
\end{align}
with equality precisely when $S$ is empty and $1$ is available for every vertex. That is,
\begin{align}
    \E_C[X_1\ |\ X_1 > 0]
    \leq \sum_{S\subseteq V(H)} {\mathbb P}_C[\chi^{-1}(2) = S\ |\ X_1 > 0] \cdot \frac{d \lambda (1+\lambda)^{d-1}}{(1+\lambda)^d - 1}
    = \frac{\lambda d (1+\lambda)^{d-1}}{(1+\lambda)^d - 1},
\end{align}
as desired.  We have equality in \eqref{cond_exp} when ${\mathbb P}_C[a_S = d\ |\ X_1 > 0] = 1$, which holds for the configurations where $1$ is available to every vertex but which have no dichromatic colourings. That is, for equality to hold in the claim $C$ must have no dichromatic colourings, and any colour which is available to some vertex $u$ must be available to every vertex (so the lists must be identical).
\end{proof}

Adding the inequalities in Claims \ref{claim:term1} and \ref{claim:term2} shows that \eqref{manipConstraint} holds for all $C$, proving optimality of $K_{d+1}$.

\subsection{Uniqueness}

\begin{lemma}
    The distribution induced by $K_{d+1}$ is the unique optimum of the LP relaxation \eqref{primalProgram}.
\end{lemma}
\begin{proof}
    Complementary slackness for our dual solution says that any optimal primal solution is supported only on configurations $C$ with identical boundary lists and no dichromatic colourings. These fall into three categories:
    \begin{description}
        \item[Case 0] $L_u = \emptyset$ for all $u$. In this case the edges of $H$ are immaterial, as none of $H$ can be coloured. This is the configuration $C_0$ above.
        \item[Case 1] $L_u = \{i\}$ for all $u$ (for $i = 1$ or $2$). The edges of $H$ are again immaterial, as every colouring of $H$ with only colour $i$ is allowed. Call this configuration $C_1$.
        \item[Case 2] $L_u = \{1,2\}$ for all $u$. In this case the prohibition on dichromatic colourings requires that $C = C_{K_{d+1}}$.
    \end{description}
    We can calculate $\alpha^v(C)$ and $\alpha^u(C)$ for each case. For Case 0 we have
    \[\alpha^v(C_0) = \frac{2\lambda}{1 + 2\lambda} \qquad\text{and}\qquad \alpha^u(C_0) = 0.\]
    For Case 1 we have
    \[\alpha^v(C_1) = \frac{\lambda + \lambda(1+\lambda)^d}{\lambda + (1+\lambda)^{d+1}} 
        \qquad\text{and}\qquad \alpha^u(C_1) = \frac{\lambda (1+\lambda)^d}{\lambda + (1+\lambda)^{d+1}}.\]
    And of course, for Case 2 we have
    \[\alpha^v(K_d) = \alpha^u(K_d) = \alpha_K.\]
    In both Case 0 and Case 1 we have $\alpha^u < \alpha^v$, so the only convex combination $q$ of the three cases giving $\sum_C q(C) \alpha^u(C) = \sum_C q(C) \alpha^v(C)$ (as is required for feasibility) is the one which puts all of the weight on $C_{K_{d+1}}$.
\end{proof}

\section{Distinct activities}
It is also natural to consider  a weighted version of the Widom--Rowlinson model with distinct activities $\lambda_1, \lambda_2$ for the two colours, so that the configuration $\chi$ is chosen according to the distribution
\begin{align*}
{\mathbb P}[\chi] &= \frac{  \lambda_1^{X_1(\chi)} \lambda_2^{X_2(\chi)} }{  P_G(\lambda_1, \lambda_2) }
\end{align*}
where the partition function is
\begin{align*}
P_G(\lambda_1,\lambda_2) &= \sum_{\chi \in  \Omega(G)}  \lambda_1^{X_1(\chi)} \lambda_2^{X_2(\chi)}.
\end{align*}
We can ask which $d$-regular graphs maximise $P(\lambda_1, \lambda_2)^{1/|V(G)|}$.
\begin{conj}
\label{conj:partitionUnequal}
For any $\lam_1, \lam_2 >0$, and any $d$-regular graph $G$,
\begin{equation}
\label{partitionstatement}
    P_G(\lam_1,\lam_2) \le  P_{K_{d+1}}(\lam_1,\lam_2)^{|V(G)|/(d+1)}.
\end{equation}
\end{conj}
Now denote by $\alpha_G^{1}(\lam_1,\lam_2)$ and $\alpha_G^{2}(\lam_1,\lam_2)$ the expected fraction of vertices of $G$ that receive colours $1$ and $2$ respectively in this model.
\begin{conj}
\label{conj:linearComb}
For any $\lam_1, \lam_2 >0$, the \emph{weighted occupancy fraction}
\begin{align*}
\overline \alpha_G(\lambda_1, \lambda_2) = \frac{ \lam_2 \alpha_G^{1} (\lam_1,\lam_2) +\lam_1 \alpha_G^{2} (\lam_1,\lam_2)}{\lam_1 + \lam_2}
\end{align*}
is maximised over all $d$-regular graphs by $K_{d+1}$.
\end{conj}

In fact, Conjecture \ref{conj:linearComb} implies Conjecture~\ref{conj:partitionUnequal}.  To see this, assume $\lam_1 \ge \lam_2$, and let $F_G(x) = \frac{1}{n} \log P_G(\lam_1- \lam_2 + x, x)$.    We  have
\begin{align*}
    \frac{1}{n} \log P_G(\lam_1,\lam_2) &= F_G(\lam_2) = F_G(0) + \int_0^{\lam_2} \frac{d F_G}{dx} (x) \, dx 
\end{align*}
$F_G(0) = \frac{1}{n} \log P_G(\lam_1 - \lam_2,0) = \log(1+ \lam_1 -\lam_2)$ for all graphs $G$, and so if we can show that for all $0\le x \le \lam_2$, $ \frac{d F_G}{dx} (x) $ is maximised when $G= K_{d+1}$, then we obtain (the log of) inequality \eqref{partitionstatement}. We compute:
\begin{align*}
    \frac{d F_G}{dx} (x) &= \frac{1}{n} \frac{ \frac{d}{dx}  P_G( \lam_1 - \lam_2 +x, x) }  { P_G( \lam_1 - \lam_2 +x, x) } \\
    &=\frac{1}{n} \frac{ \sum_\chi \frac{x X_1 +(\lam_1 -\lam_2+x)X_2 }{x(\lam_1-\lam_2+x)  } (\lam_1 -\lam_2+x)^{X_1} \cdot x^{X_2} }{ P_G( \lam_1 - \lam_2 +x, x) } \\
    &= \frac{1}{x(\lam_1-\lam_2+x)  } \frac{1}{n} \frac{ \sum_\chi(x X_1 +(\lam_1 -\lam_2+x)X_2 )(\lam_1 -\lam_2 +x)^{X_1} \cdot x^{X_2} }{ P_G( \lam_1 - \lam_2 +x, x) } \\
    &= \frac{1}{x(\lam_1-\lam_2+x)  } \left [ x \alpha_G^{(1)} (\lam_1-\lam_2+x,x) +(\lam_1-\lam_2+x) \alpha_G^{(2)} (\lam_1-\lam_2+x,x)\right ].
\end{align*}
Conjecture \ref{conj:linearComb} implies that this is maximised by $K_{d+1}$.

\end{document}